\documentclass[reqno]{amsart}
\usepackage[utf8]{inputenc} 
\usepackage{amsmath}
\usepackage{bbm}
\usepackage{mathtools}
\usepackage{hyperref}
\usepackage{amssymb}

\theoremstyle{definition}
\newtheorem{definition}{Definition}[section]

\theoremstyle{plain}

\newtheorem{theorem}{Theorem}
\newtheorem{corollary}{Corollary}

\newtheorem{lemma}{Lemma}
\newtheorem{conditions}{Conditions}

\DeclareMathOperator{\gl2}{GL_2}
\DeclareMathOperator{\sl2}{SL_2}

\DeclareMathOperator{\vol}{vol}

\title{Certification of Maass cusp forms of arbitrary level and character}
\author{Kieran Child \\ University of Bristol}
\email{kieran.child@bristol.ac.uk}
\address{University of Bristol \\ Beacon House \\ Queens Road \\ Bristol \\ BS8 1QU \\ UK}

\begin{document}
\begin{abstract}
We present a method for certifying the existence of an arbitrary Maass cusp form for any level and character. This is accomplished by producing a bound on the difference between the $\Delta$-eigenvalue of an authentic Maass cusp form and a purported approximation of a $\Delta$-eigenvalue, arrived at by any means. We apply this method to a proposed non-CM level 5 form with quadratic character, to present the first certified $\Delta$-eigenvalue of such a form.

This work generalises the method for certifying level 1 forms presented by Booker, Str\"ombergsson and Venkatesh, and is motivated by the production of purported Maass cusp forms of arbitrary level and character via methods developed by Hejhal and Str\"omberg.
\end{abstract}

\maketitle

\section{Introduction}

In \cite{Maass}, Maass produced the first known examples of non-holomorphic analogues of modular forms, now called Maass forms. Further work by Selberg led to a famed proof of the infinitude of Maass cusp forms in \cite{Selberg}, via the Selberg trace formula. In particular, this trace formula proved infinitely many Maass forms which had not been constructed by algebraic means such as those used by Maass. Along with holomorphic modular forms, Maass forms of weight 0 and 1 account for all automorphic forms for GL(2) (see \cite{GoldfeldAutomorphic}).

This trace formula has since been used to provide certified approximations of Maass cusp forms for level 1 in \cite{BookerStrombergsson}, with this approach being generalised to square-free level and trivial character in \cite{Andrei}. Thus far, rigorous computation of arbitrary forms has only been possible in those settings.

Alternatively, heuristic approaches to approximating Maass cusp forms have been taken to a high degree of generality. In \cite{Hejhal}, an algorithm was derived for generating approximations of the Laplace eigenvalue and Fourier coefficients of any level 1 Maass cusp form. This was generalised to arbitrary level and character, and even some more unusual multiplier systems, in \cite{Stromberg} and \cite{StrombergWeird}. Being heuristic, these algorithms may produce false positives, and thus any approximation requires further certification.

A method for the certification of any purported level 1 Maass form approximation was presented in \cite{BookerVenkatesh}. The method provides a bound on the difference between a Laplace eigenvalue proposed by the above algorithm and the Laplace eigenvalue of an authentic Maass cusp form. In this paper, we generalise this approach to arbitrary level and character. Subsequently, a purported Maass cusp form of any level and character can be certified, regardless of the method by which it was found. We apply the approach to present a certified approximation of the Laplace eigenvalue of a Maass cusp form with quadratic character - the first such example of a certified Maass cusp form with non-trivial character known to the author.

The structure of the paper is as follows. Section \ref{SectionPrereq} covers the prerequisites leading to the presentation of the results. Theorem \ref{maintheorem} is an expression bounding the difference between the Laplace eigenvalues of a purported and authentic Maass cusp form. Corollary \ref{theform} presents a certified form with level 5 and quadratic character. Section \ref{SectionProof} constitutes a proof of Theorem \ref{maintheorem}, and finally Section \ref{SectionPractical} covers the practical computation, leading to the proof of Corollary \ref{theform}.

\section{Certification of Maass cusp forms}
\label{SectionPrereq}
\subsection{Preliminary theory}
Let $\mathbb{H} = \{z \in \mathbb{C} : \Im(z)>0\}$ be the Poincar\'e upper half-plane. By writing points on the plane as $z=x+iy$, we equip the metric:
\begin{equation}
ds^2=\frac{dx^2+dy^2}{y^2},
\end{equation}
and volume element:
\begin{equation}
d\mu = \frac{dxdy}{y^2}.
\end{equation}
The hyperbolic distance between two points in $\mathbb{H}$ is given by:
\begin{equation}
d(z,w)=\log \left( \frac{|z-\overline{w}|+|z-w|}{|z-\overline{w}|-|z-w|}\right).
\end{equation}
The group $\gl2(\mathbb{R})^+$ acts on $\mathbb{H}$ by M\"obius transformations:
\begin{equation}
\gamma z = \frac{az+b}{cz+d},\;\; \forall \gamma = \begin{pmatrix}
a & b \\ c & d
\end{pmatrix} \in \gl2(\mathbb{R})^+.
\end{equation}
We see that hyperbolic distance is invariant when replacing $(z,w)$ with $(\gamma z, \gamma w)$ for any $\gamma \in \gl2(\mathbb{R})^+$.
For any subgroup $\Gamma \leqslant \gl2(\mathbb{R})^+$ we say two points $z_1, z_2 \in \mathbb{H}$ are $\Gamma-$equivalent if there exists some $\gamma \in \Gamma$ such that $\gamma z_1=z_2$. Let $\Gamma_0(N)$ be the principal congruence subgroup:
\begin{equation}
\Gamma_0(N) = \left\{ \begin{pmatrix}
a & b \\ c & d
\end{pmatrix} \in \sl2(\mathbb{Z}): c \equiv 0 \pmod{N} \right\}.
\end{equation}
When working with $\Gamma_0(1)=\sl2(\mathbb{Z})$ we use the generators:
\begin{equation}
\label{generators}
T=\begin{pmatrix}
1 & 1 \\ 0 & 1
\end{pmatrix},\;\;\; S= \begin{pmatrix}
0 & -1 \\ 1 & 0 
\end{pmatrix}.
\end{equation}
Let $\chi$ be a Dirichlet character. We denote the conductor of  $\chi$ by $\mathfrak{f}(\chi)$, and extend the domain to $\sl2(\mathbb{Z})$ by setting:
\begin{equation}
\chi(\gamma)=\overline{\chi(a)}=\chi(d),\;\; \forall \gamma = \begin{pmatrix}
a & b \\ c & d
\end{pmatrix}\in \sl2(\mathbb{Z}).
\end{equation}

A function $f: \mathbb{H} \rightarrow \mathbb{C}$ is said to be automorphic with level $N$ and character $\chi$ if it satisfies the following relation for all $\gamma \in \Gamma_0(N)$:
\begin{equation}
f(\gamma z) = \chi(\gamma)f(z).
\end{equation}

These relations are called automorphy relations. Considering the automorphy relation with $\gamma = \begin{pmatrix}
1 & 0 \\ N & 1
\end{pmatrix}\begin{pmatrix}
1 & 1 \\ 0 & 1
\end{pmatrix}$, we see that for any non-zero automorphic function, the conductor of $\chi$ must divide $N$. Similarly, by considering $\begin{pmatrix}
-1 & 0 \\ 0 & -1
\end{pmatrix}$ we see that $\chi$ must be even. We now define a subset of points on $\mathbb{H}$ on which the values of an automorphic function completely characterise that function.

\begin{definition}
Fix a level $N$. A fundamental domain for $\Gamma_0(N)$ is an open set $F$ of points in $\mathbb{H}$ such that:
\begin{itemize}
\item No two distinct points in $F$ are $\Gamma_0(N)$-equivalent.
\item Every point in $\mathbb{H}$ is $\Gamma_0(N)$-equivalent to at least one point in the closure of $F$.
\end{itemize}
The openness (and subsequent closure) of $F$ are with regards the $\mathbb{C}$ topology.
\end{definition}
This definition differs slightly from that found in many sources (for example \cite[Section 2.2]{IwaniecMethods}) in that we do not require $F$ to be connected. For a composite $N \ge 6$ our fundamental domain will not actually be a domain. 

Any level $N$ automorphic function is defined entirely by its values on the closure of the fundamental domain for $\Gamma_0(N)$, and so we construct an inner product by evaluating on these points. For two such functions $f$ and $g$ we define the Petersson inner product:
\begin{equation}
\label{PeterssonInner}
\langle f, g \rangle = \int_{F} f(z)\overline{g(z)} d\mu.
\end{equation}
Letting $||f||_2 = \langle f, f \rangle^{\frac{1}{2}}$, denoting all automorphics functions of level $N$ and character $\chi$ by $\mathcal{A}(N,\chi)$, we obtain a Hilbert space:
\begin{equation}
L^2(N,\chi) = \{f \in \mathcal{A}(N,\chi) : ||f||_2 < \infty \}.
\end{equation}

Following \cite{IwaniecMethods}, we decompose $L^2(N,\chi)$ by way of the hyperbolic Laplacian:
\begin{equation}
\Delta = y^2 \left(\frac{\partial^2}{\partial x^2}+\frac{\partial^2}{\partial y^2}\right).
\end{equation}
Let $\mathcal{E}(N,\chi)$ denote the space spanned by the continuous spectrum of $\Delta$, and let $\mathcal{S}(N,\chi)$ be its orthogonal complement with respect to (\ref{PeterssonInner}). Then for any level $N$ and character $\chi$ we have:
\begin{equation}
\label{eqdecomp}
L^2(N,\chi) = \mathbb{C} \oplus \mathcal{E}(N,\chi) \oplus \mathcal{S}(N,\chi),
\end{equation}
where $\mathbb{C}$ is only present if $\chi=\mathbbm{1}$.

The space $\mathcal{E}(N,\chi)$ is spanned by \emph{incomplete Eisentein series}. This space is studied in detail in \cite{Young}, and explicit formulae for the Fourier expansions of basis elements are given. The space $\mathcal{S}(N,\chi)$ is more mysterious. Forms in this space are cuspidal, meaning $f(z) \rightarrow 0$ as $z \rightarrow z' \in \mathbb{Q} \cup \{\infty\}$, and while their existence has been proven, no such explicit formulae exist for their Fourier expansions.

An $L^2$, automorphic $\Delta-$eigenfunction is called a Maass form. For any $\lambda \in \mathbb{C}$, we define the Maass cusp form space:
\begin{equation}
\mathcal{S}(N,\chi,\lambda) = \{f \in \mathcal{S}(N,\chi) : \Delta f = \lambda f \}.
\end{equation}
These spaces are finite-dimensional, and are non-zero only if $\lambda \in \mathbb{R}_{>0}$. It is further conjectured (see \cite{SelbergConjecture}) that $\mathcal{S}(N,\chi,\lambda)=0$ whenever $\lambda < \frac{1}{4}$. 

Let $f(z) \in \mathcal{S}(N,\chi,\lambda)$ be a Maass cusp form. It follows from the definition that $f(Bz) \in \mathcal{S}(M,\chi,\lambda)$ for any $M\in \mathbb{N}$ such that $N\mid M$ and $B \mid \frac{M}{N}$. The subspace of $\mathcal{S}(M,\chi,\lambda)$ spanned by forms arising in this way from levels less than $M$ is called the oldform space, denoted $\mathcal{S}^{\text{old}}(M,\chi,\lambda)$, and its orthogonal complement with respect to (\ref{PeterssonInner}) is called the newform space, denoted $\mathcal{S}^{\text{new}}(M,\chi,\lambda)$. Certification of all functions in $L^2(N,\chi)$ then follows from the certification of forms in $\mathcal{S}^{\text{new}}(M,\chi,\lambda)$.

Following \cite{HejhalTrace}, we give a generic Fourier expansion for any Maass form. A \emph{cusp} is any point in $\mathbb{Q} \cup \{\infty\}$, with $\infty$ understood to be a point infinitely far up the imaginary axis, and taken to be $\frac{1}{0}$ when treated as a fraction. We will consider the Fourier expansions at multiple cusps.

The width $h$ of a cusp $\mathfrak{a}=\frac{a}{b}$ is given by $\frac{N}{(N,b^2)}$. A cusp-normalising map is any element $\sigma \in \sl2(\mathbb{R})$ such that $\sigma(\infty)=\mathfrak{a}$ and $\sigma T \sigma^{-1}$ is a generator of the subgroup of $\Gamma_0(N)$ which stabilises $\mathfrak{a}$. As in \cite{Stromberg}, we can always take:
\begin{equation}
\sigma_{\mathfrak{a}} = \begin{pmatrix} a & \cdot \\ b & \cdot \end{pmatrix}\begin{pmatrix}
\sqrt{h} & 0 \\ 0 & \sqrt{h}^{-1}
\end{pmatrix},
\end{equation}
where the first matrix is in $\sl2(\mathbb{Z})$. We use the notation $e(x)=e^{2 \pi i x}$. Given a character $\chi$, we define the cusp parameter $\mu_{\mathfrak{a}} \in [0,1)$ by:
\begin{equation}
\chi(\sigma_{\mathfrak{a}} T \sigma_{\mathfrak{a}}^{-1}) = e(\mu_{\mathfrak{a}}).
\end{equation}
For any $\lambda \in \mathbb{R}_{\ge \frac{1}{4}}$, fix $r=\sqrt{\lambda-\frac{1}{4}}$, and define the Whittaker function:
\begin{equation}
W_{ir}(y) = K_{ir}(y)\sqrt{y},
\end{equation}
where $K$ is the $K$-bessel function. We now state the generic Fourier expansion. Let $f \in \mathcal{S}(N,\chi,\lambda)$ be a Maass cusp form. Then for any cusp $\mathfrak{a}$ we define $f_{\mathfrak{a}}(z)=f(\sigma_{\mathfrak{a}} z)$, and $f_{\mathfrak{a}}$ has an expansion of the form:
\begin{equation}
\label{FourierExpansionForm}
f_{\mathfrak{a}}(z) = \sum_{n\in \mathbb{Z} - \{\mu_\mathfrak{a}\}}\frac{a(\mathfrak{a},n)}{\sqrt{2 \pi |n+\mu_\mathfrak{a}|}} W_{ir}(2 \pi |n+\mu_{\mathfrak{a}}|y)e((n+\mu_{\mathfrak{a}})x).
\end{equation}

A more general version of this expression, including terms for non-cuspidal Maass forms, is drived in \cite{HejhalTrace}. The growth of the $a(\infty,n)$ coefficients is conjecturally bounded by the Ramanujan-Petersson conjecture (see \cite{KimSarnak}), and unlike modular forms they are in many cases believed to be transcendental. A Maass cusp form is said to have \emph{complex multiplication}, and is called a CM form, if there exists some Dirichlet character $\psi$ such that $a(\infty,p)=\psi(p)a(\infty,p)$ for all but finitely many primes $p$. One can construct CM forms from representations of GL$(1)$. See \cite{Bump} for more on this, as part of a more detailed study of Maass forms from a representation theoretic point of view.

For non-CM forms, there is no such known algebraic construction, and so we are reliant on approximations of coefficients by numerical methods. There are, however, easily expressed relations between the Fourier coefficients in the expansion at infinity of a Maass cusp form, which apply regardless of whether the form is a CM form. These relations are demonstrated with Hecke operators. For any $n \in \mathbb{N}$ and character $\chi$ we define the Hecke operator:
\begin{equation}
(T_n^{\chi}f)(z) = \frac{1}{\sqrt{n}}\sum_{ad=n}\chi(a)\sum_{b \pmod{d}}f \left(\frac{az+b}{d}\right).
\end{equation}

One sees that for any $n \in \mathbb{N}$, $T_n^\chi$ fixes the spaces $\mathcal{S}(N,\chi,\lambda)$, $\mathcal{S}^{\text{old}}(N,\chi,\lambda)$ and $\mathcal{S}^{\text{new}}(N,\chi,\lambda)$. An automorphic function which is an eigenfunction of $T_n^{\chi}$ for all $n \in \mathbb{N}$ is called a Hecke eigenfunction. From \cite{AtkinLehner}, we can construct a basis for any newform space from Hecke eigenfunctions. Let $\nu_p(n)$ be the maximal $e \in \mathbb{Z}$ such that $p^e \mid n$. Following the multiplicativity of Hecke operators, we obtain the following relations on the Fourier coefficients of said eigenfunctions:
\begin{equation}
\label{heckerelation}
a(\infty,\pm n)=a(\infty, \pm 1) \prod_{p \mid n}a(\infty, p^{\nu_p(n)}),
\end{equation}
and for $e>1$:
\begin{equation}
a(\infty, p^e) = a(\infty, 1) \left( a(\infty, p)a(\infty, p^{e-1})-\chi(p)a(\infty,p^{e-2})\right).
\end{equation}
We can normalise any Hecke eigenfunction by dividing by $a(\infty, 1)$, whereupon $T_n^{\chi}f=a(\infty,n)f$. Similar relations for expansions at other cusps is given in \cite[Theorem 19]{LeeGoldfeld}. In particular, if the maximal odd divisor of $N$ is square-free, and $8 \nmid N$, then the coefficients of the Fourier expansions of a normalised Hecke eigenform at every cusp are multiplicative. These relations are called the Hecke relations.

In \cite{Asai}, further relations are given between Fourier expansions at different cusps $\frac{1}{d}$ for any $d \mid N$ with $(d,N/d)=1$. These emerge from relating the Fourier expansions with Atkin-Lehner involutions (see \cite{AtkinLi}). Consequently, Fourier expansions at these cusps can be recovered solely from the Fourier expansion at $\infty$, and for a normalised Hecke eigenform,  $a\left(\frac{1}{d},1\right)$ must have absolute value 1.

Given the automorphy relations, we are only interested in Fourier expansions at cusps in the closure of the fundamental domain. We now specify a particular fundamental domain such that no two cusps in its boundary are $\Gamma_0(N)-$equivalent. For $N=1$, we use the ubiquitous fundamental domain for $\sl2(\mathbb{Z})$:
\begin{equation}
\mathcal{F} = \left\{ z \in \mathbb{H} : |z|>1, |\Re(z)|<\frac{1}{2}\right\}.
\end{equation}
For $N>1$, we construct a fundamental domain by first generating $A$, a set of right coset representatives for $\Gamma_0(N)\backslash \sl2(\mathbb{Z})$, in the following algorithm:
\begin{enumerate}
\item Let $c$ range over the divisors of $N$. Define $v := \gcd(N/c,c)$.
\item Let $a$ range over $(\mathbb{Z}/v\mathbb{Z})^*$, taking representatives in $\mathbb{Z}_{\ge 0}$ which are coprime to $c$.
\item Let $d$ range over all values $-1 \le d<\frac{N}{v}-1$ such that if $a \neq 0$ then $ad \equiv 1 \pmod{c}$.
\item Append the matrix $M := \begin{pmatrix} a & \frac{ad-1}{c} \\ c & d \end{pmatrix}$ to the set $A$. Here, $c$ is taken modulo $N$, so that $M=I$ if $c=N$.
\end{enumerate}

For any $M \in \sl2(\mathbb{Z})$, let $\mathcal{F}_M$ denote the image of the closure of $\mathcal{F}$ under transformation by $M$. From lemmas \ref{Arightcoset} and \ref{Adistinctcusps}, it follows that the union of $\mathcal{F}_M$ across all $M \in A$ is a fundamental domain for $\Gamma_0(N)$, and the set of $\frac{a}{b}$ where $(a;b)$ are the left columns of matrices in $A$ is a maximal set of $\Gamma_0(N)$-inequivalent cusps. For any $M \in A$, we say $\frac{a}{b}$ is the cusp associated with $M$.

%We can use this fundamental domain $F$ to construct an automorphic function $\tilde{f}$ from any given function $f$. For any $z \in \mathbb{H}$, we define the pullback matrix of $z$ as any $\rho \in \Gamma_0(N)$ sending $z$ into the closure of $F$. So long as $z$ is not $\Gamma_0(N)-$equivalent to the boundary of $F$, then it has a uniquely defined pullback matrix. For any function $f:\mathbb{H} \rightarrow \mathbb{C}$, if we define $\tilde{f}$ by $\tilde{f}(z)=\overline{\chi(\rho)}f(\rho z)$, with some standardised selection of $\rho$ on points $\Gamma_0(N)$-equivalent to the boundary, then $\tilde{f}$ is automorphic.

\subsection{Results}
Given a purported approximation of a Maass cusp form, we produce a bound on the difference between the purported $\Delta-$eigenvalue and the $\Delta-$eigenvalue of an authentic Maass cusp form. We assert the following conditions on our input.

\begin{conditions}
\label{conditions1}
We require a fixed level $N \in \mathbb{N}$, Dirichlet character $\chi$ (defined mod $N$), and $\lambda \in \mathbb{R}_{\ge \frac{1}{4}}$. We also require a full set of functions $f_{\mathfrak{a}}$ attached to the cusps of the fundamental domain for $\Gamma_0(N)$, and that each of these functions:
\begin{itemize}
\item is a truncated Fourier expansion matching the form in (\ref{FourierExpansionForm});
\item satisfies the Hecke relations given in \cite[Theorem 19]{LeeGoldfeld};
\item satisfies the relationships between cusps given in \cite{Asai}.
\end{itemize}
Finally, we require that $f_{\infty}$ is normalised so that $a(\infty,1)=1$.
\end{conditions}

It may appear that these conditions are very strict. For example, at level 6, we appear to need approximations of four distinct Fourier expansions, all of which are expected to have transcendental coefficients. In actuality, the requirement that the relationships in \cite{Asai} be satisfied entails that all such Fourier expansions could be recovered by the eigenvalues of Atkin-Lehner involutions, along with the expansion at $\infty$. At level 6, we would only require $a(\infty,2)$, $a(\infty,3)$, $a(\infty,5)$, $a(\infty,7)$ and $a(\infty,-1)$ to determine all the $f_\mathfrak{a}$ functions between the $-10$th and 10th coefficients.

We now construct an automorphic function $\tilde{f}$, which will be used to state Theorem \ref{maintheorem}. Suppose we have $N, \chi, \lambda$ and $f_\mathfrak{a}$ satisfying Conditions \ref{conditions1}. For any $z \in \mathbb{H}$, let $\rho$ be the pullback matrix of $z$, and let $M \in A$ be such that $\rho z \in \mathcal{F}_M$. Let $\mathfrak{a}$ be the cusp associated with $M$. We define:
\begin{equation}
\label{tildedef}
\tilde{f}(z)=\overline{\chi(\rho)} \begin{cases}
f_\infty(\rho z) & \text{if }N \ge 3, \mathfrak{a}=0 \text{ and } |\rho z| \ge \frac{1}{\sqrt{3}},\\
f_\mathfrak{a}(\sigma^{-1}\rho z) & \text{else}.
\end{cases}
\end{equation}

Note that $\tilde{f}$ is poorly defined on points which are $\Gamma_0(N)$-equivalent to points on the arc $\{z=e(\theta):\frac{1}{6} \le \theta \le \frac{1}{3}\}$. This will not matter, as the arc has volume 0 and Theorem \ref{maintheorem} only calls on $\tilde{f}$ in the context of evaluating an essential supremum.

\begin{theorem}
\label{maintheorem}
Suppose we have $N$, $\chi$, $\lambda$ and a full set of $f_\mathfrak{a}$ meeting Conditions \ref{conditions1}. Fix some $\delta \in \mathbb{R}_{>0}$ such that $\delta < \Im(z)$ for all $z=Me(1/6)$ and $z=Me(1/3)$ with $M \in A$. Let $B(\delta)$ be the neighbourhood of points with hyperbolic distance at most $\delta$ from the arc $\{z=e(\theta): \frac{1}{6} \le \theta \le \frac{1}{3} \}$, and let $B'(\delta)$ be the neighbourhood of points with hyperbolic distance at most $\delta$ from the arc $\{z = \frac{e(\theta)}{\sqrt{3}}: \frac{1}{12}\le \theta \le \frac{5}{12}  \}$. Fix $m \in \mathbb{Z}_{>1}$ such that $(m,N)=1$ and let $\Psi$ be any maximal set of Dirichlet characters $\psi$ satisfying the following criteria:
\begin{itemize}
\item $\mathfrak{f}(\psi)\mathfrak{f}(\chi\overline{\psi}) \mid N$;
\item $\psi_1, \psi_2 \in \Psi$ only if $\psi_1(m) \neq \psi_2(m)$.
\end{itemize}

There exists a Maass cusp form with Laplace eigenvalue $\tilde{\lambda}$ satisfying:
\begin{equation}
|\tilde{\lambda}-\lambda| \le \frac{40 (N(T_m)+2)^{\# \Psi}  \sqrt{\max(\{E(M): M \in A\})}}{\delta^{\frac{3}{2}} \prod_{\psi \in \Psi}\left(a(\infty,m) - \sum_{ab=m}\psi\left(\frac{a}{b}\right)\chi(b)\left(\frac{a}{b}\right)^{ir}\right) \sqrt{D}},
\end{equation}
where the details of specific functions are as follows. Firstly, $D$ is defined as:
\begin{align}
\label{Dequation}
D =& 2^{\omega(N)}\int_{m^{\#\Psi}e^\delta}^\infty W_{ir}(2\pi y)^2 \frac{dy}{y^2}\nonumber\\ 
&+\int_{\frac{m^{\#\Psi}e^\delta}{\sqrt{3}}}^{m^{\#\Psi}e^{\delta}} W_{ir}(2\pi y)^2 \frac{dy}{y^2}\\
&-\int_{m^{\#\Psi}e^\delta}^{\sqrt{3} m^{\#\Psi}e^{\delta}} W_{ir}(2\pi y)^2 \frac{dy}{y^2},\nonumber
\end{align}
where $\omega(N)$ is the number of prime factors of $N$.
For the $E$ function, let $\mathfrak{a}$ be the cusp associated with the matrix $M$, and let $||f||_{\infty, B}$ be the essential supremum of $f$ over the open neighbourhood $B$. We define:
\begin{equation}
E(M) = \begin{cases}
||\tilde{f}-f_{\mathfrak{a}}||_{\infty, MB(\delta)}^2& \text{if }M \not \in \{I, S, ST, ST^{-1}\} \text{ or }N<3,\\
3||\tilde{f}-f_{\mathfrak{a}}||_{\infty, B'(\delta)}^2& \text{if }M^2=I \text{ and }N \ge 3,\\
0 & \text{else.}
\end{cases}
\end{equation}
Finally, $N(T_m)$ is defined by $N(T_1)=1$ along with the following relation for $p \mid m$:
\begin{equation}
N(T_m) = N(T_{\frac{m}{p}})(p^{\frac{7}{64}}+p^{-\frac{7}{64}}) + \underset{p^2 \mid m}{\delta}N(T_{\frac{m}{p^2}}).
\end{equation}
\end{theorem}

The function $E$ evaluates how well the provided $f_\mathfrak{a}$ satisfy the automorphy relations one expects of a Maass cusp form. If $\lambda$ and $f_{\mathfrak{a}}$ are good approximations of an actual Maass cusp form, we expect the resultant bound to be relatively small, certifying the $\Delta-$eigenvalue.

We could alternatively have defined $\tilde{f}$ on all of $\mathbb{H}$ as $\overline{\chi(\rho)}f_{\mathfrak{a}}(\sigma^{-1}\rho z)$. This would be a more natural definition, but would result in the expression for $D$ in (\ref{Dequation}) consisting solely of the first term, producing a worse bound.

As $W_{ir}(y) \sim e^{-y}$, we obtain the best bound by choosing $m$ such that $m^{\#\Psi}$ is minimised. For any $p \mid N$, let $e=\nu_p(N)$ and $s=\nu_p(\mathfrak{f}(\chi))$. We take $m$ to be the least integer satisfying all the following criteria:
\begin{itemize}
\item $m\ge 2$;
\item $(m,N)=1$;
\item For all $p \mid N$, $m \equiv 1 \pmod{p^{\min(\lfloor \frac{e}{2} \rfloor, e-s)}}$.
\end{itemize}
In particular, when $\mathfrak{f}(\chi)=N=p^e$ for $p>2$ we take $m=2$, whereupon $\# \Psi(N,\chi,m)=2$ and the lowest limit on the integrals defining $D$ is $\frac{4 e^\delta}{\sqrt{3}}$.

Applying this theorem to the purported Maass cusp form approximation given in Appendix \ref{appendix} gives the corollary:

\begin{corollary}
\label{theform}
There exists a Maass cusp form $f \in \mathcal{S}(5,\left(\frac{\cdot}{5}\right),\lambda)$, with:
\begin{equation}
|\lambda-24.199|<10^{-2}.
\end{equation}
\end{corollary}

\section{Proof of Theorem \ref{maintheorem}}
\label{SectionProof}

We begin by proving that our generated set $A$ does indeed give a fundamental domain with the properties we claimed.

\begin{lemma}
\label{Arightcoset}
$A$ is a complete set of right coset representatives for $\Gamma_0(N)\backslash SL_2(\mathbb{Z})$.
\end{lemma}
\begin{proof}
Firstly, we prove that all matrices in $A$ are in $SL_2(\mathbb{Z})$. Clearly, all matrices in $A$ have determinant 1, and so it only remains to prove that $\frac{ad-1}{c}$ is integral. When $c=1$, this is immediate. When $c>1$, $ad \equiv 1 \pmod{c}$ and so $\frac{ad-1}{c}$ is integral.

Next we prove that any matrix in $SL_2(\mathbb{Z})$ has a coset representative in $A$. Fix $M_1=\begin{pmatrix} a & b \\ c & d \end{pmatrix} \in SL_2(\mathbb{Z})$. If $N \mid c$ then $M_1$ is in the same coset as $I \in A$. Else, fix $c'=(N,c)$ and $v=\gcd(N/c',c')$. Let $a'$ be the value chosen in step 2 satisfying $a' \equiv \frac{ac}{c'} \pmod{v}$. Define $d'$ by the congruences $d' \equiv \frac{c'd}{c} \pmod{\frac{N}{c'}}$ and $d' \equiv \frac{1}{a'} \pmod{c'}$, with the second congruence only applying if $c'>1$. These conditions are conformal mod $v$ and so define $d'$ mod $\frac{N}{v}$, matching a value from step 3.

We thus have an element $M_2=\begin{pmatrix}
a' & \frac{a'd'-1}{c'} \\ c' & d'
\end{pmatrix} \in A$, and multiplying $M_1$ by $M_2^{-1}$ gives a matrix with lower left entry $c'd-d'c$. As $d' \equiv \frac{c'd}{c} \pmod{\frac{N}{c'}}$, this entry is a multiple of $N$, and so $M_1$ and $M_2$ are representatives of the same coset.

Finally, we prove that no two distinct elements in $A$ represent the same coset. Let $M_1, M_2$ be two elements in $A$ with bottom rows $(c_1,d_1)$ and $(c_2,d_2)$ respectively. Suppose $M_1M_2^{-1} \in \Gamma_0(N)$, then $c_1d_2 \equiv c_2d_1 \pmod{N}$. As $c_1$ and $d_1$ are coprime, and $c_1 \mid N$, we must have that $c_1 \mid c_2$. Applying this argument a second time to $c_2$ and $d_2$ gives $c_1=c_2$, and consequently $d_1 \equiv d_2 \pmod{N/c_1}$. Letting $v=\gcd(\frac{N}{c_1},c_1)$, this fixes the upper left entry $a$ in both matrices. As $d_1a \equiv d_2a \equiv 1 \pmod{c_1}$ we deduce that $d_1 \equiv d_2 \pmod{\frac{N}{v}}$, whereupon the inequality $0 \le d_1, d_2 < \frac{N}{v}$ necessitates that $d_1=d_2$ and so $M_1=M_2$.
\end{proof}

Next, we show that the closure of the fundamental domain constructed from $A$ does not contain $\Gamma_0(N)$-equivalent cusps. 

\begin{lemma}
\label{Adistinctcusps}
No two elements in $A$ map $\infty$ to distinct, $\Gamma_0(N)$-equivalent cusps.
\end{lemma}
\begin{proof}
Fix two elements in $A$:
\begin{equation}
M_1 = \begin{pmatrix}
a_1 & b_1 \\ c_1 & d_1
\end{pmatrix},\;\;\;\; M_2 = \begin{pmatrix}
a_2 & b_2 \\ c_2 & d_2
\end{pmatrix}.
\end{equation}
Note that these map $\infty$ to $\frac{a_1}{c_1}$ and $\frac{a_2}{c_2}$ respectively. Assume that these cusps are $\Gamma_0(N)$-equivalent. Then there exists an $n \in \mathbb{Z}$ such that:
\begin{equation}
\begin{pmatrix}
a_1 & b_1 \\ c_1 & d_1
\end{pmatrix}\cdot \begin{pmatrix} 1 & n \\ 0 & 1 \end{pmatrix} \cdot \begin{pmatrix}
a_2 & b_2 \\ c_2 & d_2
\end{pmatrix}^{-1} \in \Gamma_0(N).
\end{equation}
Evaluating the lower left entry of this matrix, we see that $N \mid c_1d_2 - c_2d_1 - nc_1c_2$. As $c_1$ and $d_1$ are coprime, and $c_1 \mid N$, we must have that $c_1 \mid c_2$. Applying this argument a second time to $c_2$ and $d_2$ gives $c_1=c_2$. Consequently, we deduce $c_1d_2 \equiv c_1d_1 \pmod{(N,c_1^2)}$.

Dividing this through by $c_1$ gives $d_1 \equiv d_2 \pmod{v}$, where $v=(c_1,N/c_1)$. From step 3 of the algorithm, we see that $a_1 \equiv a_2 \pmod{v}$, but as we only have a single representative of each class in $(\mathbb{Z}/v\mathbb{Z})^*$ we conclude $a_1=a_2$.
\end{proof}

We now move to constructing a function in the span of Maass cusp forms. Assume a given $N, \chi, \lambda$ and full set of $f_\mathfrak{a}$ meeting Conditions \ref{conditions1} and let $\tilde{f}$ be as in (\ref{tildedef}). We construct a function $\tilde{f}_S \in L^2(N,\chi)$ by convolution with a point-pair invariant, as outlined in \cite[Section 1.8]{IwaniecMethods}. For any function $k: \mathbb{H} \times \mathbb{H} \rightarrow \mathbb{C}$ we define $f\star k$ as the convolution:
\begin{equation}
f \star k (z) = \int_{\mathbb{H}} k(z,w)f(w) d\mu(w).
\end{equation}
Fix some $\delta \in \mathbb{R}_{>0}$, and let $k$ be an $L^2$ function, dependant only on the hyperbolic distance between $z$ and $w$, which is 0 whenever this distance is greater than $\delta$. Defining $\tilde{f}_S = \tilde{f} \star k$, we see that $\tilde{f}_S \in L^2(N,\chi)$ and so by (\ref{eqdecomp}) we obtain a function in the span of Maass cusp forms by constructing an operator which annihilates $\mathbb{C}$ and $\mathcal{E}(N,\chi)$.

\begin{definition}[Diamond operator]
Fix a level $N$, Dirichlet character $\chi$, and $m \in \mathbb{Z}_{>1}$ such that $(m,N)=1$. Let $\Psi$ be as in Theorem \ref{maintheorem}. We define the operator:
\begin{equation}
\diamondsuit := \prod_{\psi \in \Psi} \left(T_m - \sum_{ab=m} \psi(a/b)\chi(b)(b/a)^{\sqrt{\frac{1}{4}-\Delta}}\right).
\end{equation}
\end{definition}
This is a generalisation of the $\aleph$ operator introduced in \cite{Venkatesh} for the purpose of annihilating Eisenstein series and constants for level 1.
\begin{lemma}
For any $f \in L^2(N,\chi)$, $\diamondsuit (f) \in \mathcal{S}(N,\chi)$.
\end{lemma}
\begin{proof}
The component functions $T_m$ and $\Delta$ stabilise $L^2(N,\chi)$, and thus $\diamondsuit(f)\in L^2(N,\chi)$. The continuous spectrum is, by \cite{Young}, spanned by Eisenstein series attached to characters $\psi$, having:
\begin{align}
\Delta E_\psi =& \left(\frac{1}{4}+r^2\right)E_\psi, \\
T_m E_\psi =& \left(\sum_{ab=m}\psi(a/b)\chi(b)(b/a)^{ir}\right)E_\psi.\\
\end{align}
As $\diamondsuit$ sends any such series to 0, it annihilates $\mathcal{E}(N,\chi)$. 

If $\chi=\mathbbm{1}$, consider the factor when $\psi=\mathbbm{1}$ acting on constants. Both $T_m$ and $\sum (b/a)^{\sqrt{\frac{1}{4}-\Delta}}$ multiply any constant by $\frac{\sigma(m)}{\sqrt{m}}$, and as such $\diamondsuit$ annihilates $\mathbb{C}$. The result then follows from the spectral decomposition in (\ref{eqdecomp}).
\end{proof}

Defining $g := \diamondsuit(\tilde{f}_S)$, we see that $g$ is in the span of Maass cusp forms of level $N$ and character $\chi$. It remains to deduce a bound between $\lambda$ and the $\Delta$-eigenvalue of a form contributing to $g$.

\begin{lemma}
\label{prebound}
There exists a Maass cusp form with Laplace eigenvalue $\tilde{\lambda}$ satisfying:
\begin{equation}
|\tilde{\lambda}-\lambda| \le \frac{||(\Delta-\lambda)g||_2}{||g||_2}.
\end{equation}
\end{lemma}
\begin{proof}
Let $f_j$ be a basis of eigenforms for the cuspidal spectrum with Laplace eigenvalues $\lambda_j$, and let $\epsilon_j$ be such that $g = \sum_{j=1}^\infty \epsilon_j f_j$. For any $H \in \mathbb{R}_{\ge 0}$ let $P$ be the subset of indices of Maass forms satisfying $|\lambda_j - \lambda|^2 \le H$, and define the projection $Pr_H(g)=\sum_{j \in P}\epsilon_j f_j$. Note that:
\begin{equation}
\sum_{j \in \mathbb{N}-P} | \epsilon_j |^2 \le \sum_{j=1}^\infty |\epsilon_j|^2 \frac{|\lambda_j-\lambda|^2}{H},
\end{equation}
and so:
\begin{equation}
||Pr_H(g)||_2^2 \ge ||g||_2^2 - \sum_{j=1}^\infty |\epsilon_j|^2 \frac{|\lambda_j-\lambda|^2}{H}.
\end{equation}
Thus, by setting:
\begin{equation}
H > \frac{||(\Delta-\lambda)g||_2^2}{||g||_2^2},
\end{equation}
we see that $Pr_H(g)$ is non-zero. This means that $P$ is non-empty, and so there exists a Maass form with Laplace eigenvalue $\tilde{\lambda}$ satisfying $|\tilde{\lambda}-\lambda|^2 \le H$.
\end{proof}

Theorem \ref{maintheorem} follows from finding bounds on the numerator and denominator in this lemma in terms of our original $f_\mathfrak{a}$ functions. First we find a lower bound for $||g||_2$. This is achieved by identifying a large portion of the fundamental domain on which $g$ can be written explicitly in terms of $f$. We start by determining which Fourier expansion defines $\tilde{f}(z)$ for a given $z$.

\begin{lemma}
\label{imaggivesdef}
Fix a level $N$ and divisor $d \mid N$ such that $(d, N/d)=1$. Let $\mathfrak{a}$ be the cusp $\frac{1}{d}$ with cusp-normalising map $\sigma$ and width $h$. For any $z \in \mathbb{H}$ with $\Im(\sigma^{-1} z)>\frac{1}{h}$, we have $\tilde{f}(z) = f_\mathfrak{a}(\sigma^{-1}z)$. If $N=2$ then $\tilde{f}(z)=f_0(Sz/2)$ whenever $\Im(Sz)>1$, and $\tilde{f}(z)=f_\infty(z)$ whenever $\Im(z)>1$. If $N\ge 3$ then $\tilde{f}(z)=f_0(\sigma^{-1}z)$ whenever $\Im(Sz)>\sqrt{3}$, and $\tilde{f}(z)=f_\infty(z)$ whenever $\Im(z)>\frac{1}{\sqrt{3}}$.
\end{lemma}
\begin{proof}
First we prove the statement on the cusp $\mathfrak{a}=\frac{1}{d}$. Let $B$ be the subset of $A$ consisting of matrices associated with $\mathfrak{a}$. Fix $M=\begin{pmatrix} 1 & 0 \\ d & 1 \end{pmatrix} \in B$, and note that $B = \left\{ MT^r : 0 \le r < h\right\}$. If $\Im(\sigma^{-1} z)>\frac{1}{h}$ then there exists some $n$ such that $T^nM^{-1}z \in \mathcal{F}_I$, and consequently there exists some $h \mid m$ such that $MT^mM^{-1}$ is the pullback matrix of $z$. Noting that $\chi(MT^mM^{-1})=1$, we conclude that $\tilde{f}(z)=f_{\mathfrak{a}}(\sigma^{-1}z)$.

The statements on $f_\infty$ are clear from the definition of $\tilde{f}$, and the statements on $f_0$ follow from the above, taking $M=ST^{-1}$.
\end{proof}
Using this, we show that for certain $z$, all evaluations of $\tilde{f}$ in the definition of a Hecke operator are evaluations of the same cusp's Fourier expansion. Let $|\hat{k}|$ be the eigenvalue of convolving $k$ with functions having $\Delta$-eigenvalue $\lambda$, per \cite[Theorem 1.14]{IwaniecMethods}.

\begin{lemma}
\label{HeckeAction}
Fix a level $N \ge 3$, a matrix $M= \begin{pmatrix}
0 & -1 \\ 1 & n
\end{pmatrix} \in A$ sending $\infty$ to 0, and a prime $p \nmid N$. Let $z$ be any complex number with $\Im(z) > \sqrt{3}pe^\delta$. Then:
\begin{equation}
T_p\tilde{f}_S(Mz) = |\hat{k}|\chi(p)T_p^{\overline{\chi}} f_0\left( \frac{z+n}{N}\right)
\end{equation}
\end{lemma}
\begin{proof}
First, consider the function $\tilde{f}_S\left(\frac{Mz+a}{p}\right)$ with $0<a<p$. Define $b$ by the congruences $b \equiv n(1-p) \pmod{N}$ and $b \equiv n-\overline{a} \pmod{p}$. Consider the matrix:
\begin{equation}
M'=M \begin{pmatrix}
1 & b \\ 0 & p
\end{pmatrix} M^{-1} \begin{pmatrix}
1 & a \\ 0 & p
\end{pmatrix}^{-1}.
\end{equation}
The entries of this matrix are computed as:
\begin{equation}
M' = \begin{pmatrix}
p & -a \\ n(1-p)-b & \frac{ab-an(1-p)+1}{p}
\end{pmatrix}.
\end{equation}
Using $b \equiv n-\overline{a} \pmod{p}$, the numerator of the lower right entry is seen to be a multiple of $p$, and therefore $M' \in \sl2(\mathbb{Z})$. Further, using $b \equiv n(1-p) \pmod{N}$, the lower left entry is seen to be a multiple of $N$. Consequently:

\begin{equation}
\tilde{f}_S\left(\frac{Mz+a}{p}\right)=\overline{\chi(M')}\tilde{f}_S\left(M'\begin{pmatrix}
1 & a \\ 0 & p \end{pmatrix}Mz\right) =\chi(p) \tilde{f}_S\left(M\left(\frac{z+b}{p}\right)\right).
\end{equation}
As $\Im(z) > \sqrt{3}pe^{\delta}$, $\Im\left(SM\left(\frac{z+b}{p}\right)\right) > \sqrt{3}e^{\delta}$, and so by Lemma \ref{imaggivesdef} and \cite[Theorem 1.14]{IwaniecMethods}:
\begin{equation}
\tilde{f}_S\left(M\left(\frac{z+b}{p}\right)\right) = |\hat{k}|f_0 \left( \sigma^{-1}M \left(\frac{z+b}{p}\right)\right) = |\hat{k}|f_0 \left(\begin{pmatrix}
1 & n \\ 0 & N
\end{pmatrix} \begin{pmatrix}
1 & b \\ 0 & p
\end{pmatrix} z \right).
\end{equation}
Interchanging these final two matrices gives:
\begin{equation}
\tilde{f}_S\left(\frac{Mz+a}{p}\right) = |\hat{k}|f_0 \left(\begin{pmatrix}
1 & \frac{np-\overline{a}}{N} \\ 0 & p
\end{pmatrix} \frac{z+n}{N}\right).
\end{equation}
Note that $\frac{np-\overline{a}}{N} \in \mathbb{Z}$ and, due to the periodicity of $f_0$, this matrix need only be defined up to mod $p$. The full result is arrived at by following the same process for the functions $\tilde{f}_S \left(\frac{Mz}{p}\right)$ and $\tilde{f}_S(pMz)$.
\end{proof}

Using a similar argument for any matrix $M \in A$ associated with cusp $\frac{1}{d}$ where $(d,N/d)=1$, we see that for points $z$ with $\Im(z)>pe^{\delta}$:
\begin{equation}
T_p\tilde{f}_S(Mz) = |\hat{k}|\chi_{N/d}(p)T_p^{\overline{\chi_{N/d}}\chi_d} f_\mathfrak{a}\left( \frac{z+n}{N}\right).
\end{equation}

The case that $M=I$ and $\Im(z) > \frac{pe^\delta}{\sqrt{3}}$ is immediate. By the construction of any $T_m$ from $T_p$ terms, it follows that these results still hold for a more general $T_m$. By assumption, these Fourier expansions obey Hecke relations, and the cusp parameter of all such cusps is 0. Consequently, for points in the respective parts of the fundamental domain we have:
\begin{equation}
g(z)=\sum_{n \in \mathbb{Z}-\{0\}} \frac{c(\mathfrak{a},n)}{\sqrt{2 \pi |n|}}W_{ir}(2\pi |n|y)e(nx),
\end{equation}
with:
\begin{equation}
c(\mathfrak{a},\pm 1)=\hat{k}(r)a(\mathfrak{a},\pm 1)\prod_{\psi \in \Psi(N,\chi,m)}\left(a(\infty, m)-\sum_{ab=m}\psi(a/b)\chi(b)(b/a)^{ir}\right).
\end{equation}
As noted after (\ref{heckerelation}), $f_{\infty}$ being normalised implies that $|a(\mathfrak{a},\pm 1)|=1$. By associating these cusps with the Hall divisors of $N$, we obtain the bound:
\begin{equation}
||g||_2^2 \le \frac{|c(\infty,1)|^2}{2}\sum_{d \mid N, (d,N/d)=1} \int_{m^{\#\Psi}e^\delta \ell}^\infty W_{ir}(2 \pi y)^2 \frac{dy}{y^2},
\end{equation}
where:
\begin{equation}
\ell = \begin{cases}
\frac{1}{\sqrt{3}} & \text{if }N \ge 3 \text{ and }d=N,\\
\sqrt{3} & \text{if }N \ge 3 \text{ and }d=1,\\
1 & \text{else.}
\end{cases}
\end{equation}
From \cite[Proposition 3.8]{BookerVenkatesh}, by choosing:
\begin{equation}
k(z,w)=3\left(1-\frac{|z-w|^2}{\delta^2 \Im(z)\Im(w)}\right)^2,
\end{equation}
we deduce:
\begin{equation}
\frac{1}{|\hat{k}(r)|} \le \frac{2}{\pi \delta^2}.
\end{equation}
Consequently, we obtain the bound:
\begin{equation}
\label{denominatorbound}
||g||_2 \ge \frac{11 \delta^2}{10} \prod_{\psi \in \Psi(N,\chi,m)}\left(a(\infty, m)-\sum_{ab=m}\psi(a/b)\chi(b)(b/a)^{ir}\right) \sqrt{D},
\end{equation}
where $D$ is defined as in Theorem \ref{maintheorem}. This gives the denominator of Theorem \ref{maintheorem} and we move to finding an upper bound on the $||(\Delta-\lambda)g||_2$ term appearing as the numerator in Lemma \ref{prebound}.

Note that $\diamondsuit$ commutes with $(\Delta-\lambda)$, and that the norm of $\diamondsuit$ is dependent on the norm of $T_m$. It is conjectured that the norm of $T_p$ is bounded by 2, but the best result in this direction, in \cite{KimSarnak}, bounds the norm of $T_p$ by $p^{\frac{7}{64}}+p^{-\frac{7}{64}}$. Define $N(T_m)$ by $N(T_1)=1$ along with the following relation for $p \mid m$:
\begin{equation}
N(T_m) = N(T_{\frac{m}{p}})(p^{\frac{7}{64}}+p^{-\frac{7}{64}}) + \underset{p^2 \mid m}{\delta}N(T_{\frac{m}{p^2}}).
\end{equation}
We deduce that the norm of $\diamondsuit$ is bounded above by $(N(T_m)+2)^{\#\Psi}$ and thus:
\begin{equation}
||(\Delta-\lambda)g||_2 \le (N(T_m)+2)^{\#\Psi}||(\Delta-\lambda)\tilde{f}_S||_2.
\end{equation}

To bound $||(\Delta-\lambda)\tilde{f}_S||_2^2$ we separate the $L^2$ norm into a sum over the coset representatives:

\begin{equation}
||(\Delta-\lambda)\tilde{f}_S||_2^2 = \sum_{M \in A}||(\Delta-\lambda)\tilde{f}_S||_{2, \mathcal{F}_M}^2,
\end{equation}
where $||f||_{2,\mathcal{F}_M}$ refers to the evaluation of the 2-norm over the area $\mathcal{F}_M$. Fix $M \in A-\{I, S, ST,ST^{-1}\}$, and let $\mathfrak{a}$ and $\sigma$ be the associated cusp and cusp-normalising map. Note that because $f_\mathfrak{a}(z)$ is a $\Delta$-eigenfunction with eigenvalue $\lambda$, so is $f_\mathfrak{a}(\sigma^{-1}z)$. Consequently, $(\Delta-\lambda)\tilde{f}_S = (\tilde{f}-f_\mathfrak{a}\sigma^{-1}) \star (\Delta-\lambda)k$. Let $B(\delta)$ be the neighbourhood of points with hyperbolic distance at most $\delta$ from the arc $\{z=e(\theta): \frac{1}{6}\le \theta \le \frac{1}{3}\}$. We deduce:

\begin{equation}
||(\Delta-\lambda)\tilde{f}_S||_{2,\mathcal{F}_M}^2 \le C ||\tilde{f}-f_\mathfrak{a}\sigma^{-1}||_{\infty, MB(\delta)}^2,
\end{equation}
where:
\begin{equation}
C = \vol(B(\delta) \cap \mathcal{F}_{I})\left(\int_{\mathbb{H}}|(\Delta-\lambda)k(z,i)|d\mu (z)\right)^2.
\end{equation}
If $N<3$ then the matrices $I,S,ST$ and $ST^{-1}$ are treated the same way. If $N \ge 3$ then let $B'(\delta)$ be the neighbourhood of points with distance at most $\delta$ from the arc $\{z=\frac{e(\theta)}{\sqrt{3}}: \frac{1}{12} \le \theta \le \frac{5}{12}\}$. The bound on the remaining matrices is:
\begin{equation}
\sum_{M \in \{I, S, ST, ST^{-1}\}}||(\Delta-\lambda)\tilde{f}_S||_{2,\mathcal{F}_M}^2 \le C' \max(||\tilde{f}-f_0\sigma^{-1}||_{\infty,B'(\delta)}^2,||\tilde{f}-f_\infty||_{\infty, B'(\delta)}^2),
\end{equation}
where:
\begin{equation}
C' = \vol\left(\bigcup_{M \in \{I, S, ST, ST^{-1}\}} B'(\delta) \cap \mathcal{F}_{M}\right)\left(\int_{\mathbb{H}}|(\Delta-\lambda)k(z,i)|d\mu (z)\right)^2.
\end{equation}

Following \cite[Proposition 3.8]{BookerVenkatesh} we find the bound $C \le \frac{ 2\delta \pi^2(12+1/16+4/9)^2}{\sqrt{3}}$. By the same method, let $\mathcal{F}'=\{z \in \mathbb{C} : |\Re(z)| \le \frac{1}{2}, |z| \ge \frac{e^{-\delta}}{\sqrt{3}}\}$, and note that:
\begin{equation}
\left(\bigcup_{M \in \{I, S, ST, ST^{-1}\}} B'(\delta) \cap \mathcal{F}_{M}\right) \subset \mathcal{F}' \backslash e^{\delta} \mathcal{F}'.
\end{equation}
Thus, the volume of this neighbourhood is bounded:
\begin{align}
\vol \left( \bigcup_{M \in \{I, S, ST, ST^{-1}\}} B'(\delta) \cap \mathcal{F}_{M} \right) &\le 2 \int_{x=\frac{1}{2}}^{\frac{e^{\delta}}{2}} \int_{y=\frac{e^{\delta}}{2\sqrt{3}}}^{\infty}\frac{dxdy}{y^2}\\
& = (1-e^{-\delta})2\sqrt{3} \le 2\sqrt{3}\delta.
\end{align}
Consequently, we deduce the bound $C' \le 2 \sqrt{3} \delta \pi^2 (12+1/16+4/9)^2$. Using both of these bounds gives:

\begin{equation}
\label{numeratorbound}
||(\Delta-\lambda)g||_2 \le 43 \sqrt{\delta} (N(T_m)+2)^{\#\Psi}\sqrt{ \max(\{E(M):M \in A\})},
\end{equation}
with $E$ as in Theorem \ref{maintheorem}. The theorem in full follows from applying (\ref{numeratorbound}) and (\ref{denominatorbound}) to Lemma \ref{prebound}.

\section{Practical computation}
\label{SectionPractical}

Theorem \ref{maintheorem} reduces the problem of certifying a purported Maass cusp form to one of bounding differences between the given truncated Fourier expansions over small neighbourhoods of points. We now explain how these differences are bounded in practice.

We begin with three lemmas which will identify the exact functions we need to bound to find $||\tilde{f}-f_{\mathfrak{a}}\sigma^{-1}||_{\infty, M B(\delta)}^2$. Firstly, we show that the definition of $\tilde{f}(z)$ for $z \in \mathcal{F}_M$ is in many instances given by the cusp associated with $M$, even if $M \not \in A$.

\begin{lemma}
\label{cuspgivesdef}
Let $M = \begin{pmatrix}
a & b \\ c & d
\end{pmatrix} \in \sl2(\mathbb{Z})$ be any matrix such that $\mathfrak{a}=\frac{a}{c}$ is a cusp with an associated Fourier expansion. Let $\sigma$ be the cusp-normalising map of $\mathfrak{a}$, and let $v=(c,N/c)$. For any $z \in \mathcal{F}_M$ we have $\tilde{f}(z)=\overline{\chi\left(1-\frac{aN\lfloor \frac{v(d+1)}{N}\rfloor}{v}\right)}f_{\mathfrak{a}}(\sigma^{-1}z)$.
\end{lemma}
\begin{proof}
Fix $-1 \le d' < \frac{N}{v}-1$ such that $d' \equiv d \pmod{\frac{N}{v}}$. Define the matrix:
\begin{equation}
M'=\begin{pmatrix}
a & \frac{ad'-1}{c} \\ c & d'
\end{pmatrix}.
\end{equation}
Note that $M \in \sl2(\mathbb{Z})$ implies that $(a,c)=1$. Further, either $a=0$ or $ad \equiv ad' \equiv 1 \pmod{c}$, and in both cases we have $M'\in A$. The lower left entry of $M'M^{-1}$ is $(d-d')c$ which is a multiple of $N$, and so $M'M^{-1}$ is the pullback matrix of $z$. The lower right entry is $a(d'-d)+1\pmod{N}$ and thus:
\begin{equation}
\tilde{f}(z)=\overline{\chi(M'M^{-1})}\tilde{f}(M'M^{-1}z)=\overline{\chi \left(1-\frac{aN\lfloor \frac{v(d+1)}{N}\rfloor}{v}\right)}f_{\mathfrak{a}}(\sigma^{-1}M'M^{-1}z).
\end{equation}
Fix $n=\frac{d'-d}{c}$, then $n$ is a multiple of the width of $\mathfrak{a}$ and $M'=MT^n$. Consequently:
\begin{equation}
f_{\mathfrak{a}}(\sigma^{-1}M'M^{-1}z)=f_{\mathfrak{a}}(\sigma^{-1}MT^nM^{-1}z)=f_{\mathfrak{a}}(\sigma^{-1}z).
\end{equation}
\end{proof}

Next, we reduce the area on which the difference needs to be evaluated. Assume that either $N < 3$ or $\mathfrak{a} \not \in \{0,\infty\}$. The neighbourhood $B(\delta)$ overlaps $\mathcal{F}_{M'}$ for:
\begin{equation}
M' \in \{T^{-1}, I, T, TS, TST, ST^{-1}, S, ST, T^{-1}ST^{-1}, T^{-1}S\}.
\end{equation}

We assume from now on that $\mathfrak{f}(\chi) \mid N/N_1$ where $N_1$ is the maximal value such that $N_1^2 \mid N$. We show that when evaluating the maximum of $||\tilde{f}-f_{\mathfrak{a}}\sigma^{-1}||_{\infty, M B(\delta)}^2$ across all $M$, some of these sections are considered multiple times, and so the duplicates can be discarded.

\begin{lemma}
Fix $M_1 \in A$ with associated cusp $\frac{a}{c}$ such that $\mathfrak{f}(\chi) \mid N/(c,N/c)$. For any $z \in B(\delta) \cap \left( \mathcal{F}_{TS} \cup \mathcal{F}_{TST} \cup \mathcal{F}_{T^{-1}ST^{-1}} \cup \mathcal{F}_{T^{-1}S}\right)$, there exists some $M_2 \in A$ and $z' \in B(\delta) \cap \left( \mathcal{F}_{ST^{-1}} \cup \mathcal{F}_S \cup \mathcal{F}_{ST} \right)$ such that:
\begin{equation}
\tilde{f}(M_1z)-f_{\mathfrak{a}}(\sigma^{-1}M_1z)=\tilde{f}(M_2z')-f_{\mathfrak{a}}(\sigma^{-1}M_2z').
\end{equation}
\end{lemma}
\begin{proof}
Assume $z \in B(\delta) \cap (\mathcal{F}_{TS} \cup \mathcal{F}_{TST})$. Take $z'=z-1$ and let $M_2 \in A$ be the right-coset representative of $M_1T$. We see that $M_2=M_1T$, whereupon $\tilde{f}(M_2z')=\tilde{f}(M_1z)$, unless $d+c \ge \frac{N}{v}$. In that case, by Lemma \ref{cuspgivesdef}, $\tilde{f}(M_2z')=\overline{\chi(1-\frac{aN}{v})}\tilde{f}(M_1z)$, and as $\mathfrak{f}(\chi) \mid N/v$ this is again $\tilde{f}(M_1z)$.

The distance of $z'$ from the arc is at most the distance from $e(1/6)$, which equals the distance of $z$ from $e(1/3)$, and so $z' \in B(\delta)$. If $z \in B(\delta) \cap (\mathcal{F}_{T^{-1}S} \cup \mathcal{F}_{T^{-1}ST^{-1}})$ then the same result holds, taking $z'=z+1$ and letting $M_2 \in A$ be the right-coset representative of $M_1T^{-1}$.
\end{proof}

Further, if $z \in \mathcal{F}_{T^n}$ for any $n \in \mathbb{Z}$ then $\tilde{f}(Mz)=f_{\mathfrak{a}}(\sigma^{-1}Mz)$. Consequently, we only need to bound the difference between $\tilde{f}$ and $f_{\mathfrak{a}}\sigma^{-1}$ for any $Mz$ with $z \in B(\delta) \cap (\mathcal{F}_{ST^{-1}} \cup \mathcal{F}_S \cup \mathcal{F}_{ST})$. By Lemma \ref{cuspgivesdef}, $\tilde{f}$ is defined by the same Fourier expansion for all such $z$. We now identify the exact function giving $\tilde{f}$ on these points.

\begin{lemma}
\label{M2matrix}
Fix a level $N \ge 3$ and matrix $M_1 = \begin{pmatrix}
a_1 & b_1 \\ c_1 & d_1
\end{pmatrix} \in A - \{I, S, ST, ST^{-1}\}$. Let $c_2=(d_1,N)$ and $v=(c_2,N/c_2)$. If $c_2=1$ then define $a_2=0$, else fix $a_2$ coprime to $c_2$ such that $a_2 \equiv \frac{-d_1}{c_1c_2} \pmod{v}$. Let $-1 \le d_2 < \frac{N}{v}-1$ be the element satisfying $d_2 \equiv \frac{-c_1c_2}{d_1} \pmod{N/c_2}$ and (if $c_2 \neq 1$) $a_2d_2 \equiv 1 \pmod{c_2}$. Let $M_2 \in \sl2(\mathbb{Q})$ be the matrix with entries $\begin{pmatrix}
a_2 & \cdot \\ c_2 & d_2
\end{pmatrix}$ and let $\mathfrak{b}$ be the cusp $\frac{a_2}{c_2}$. Then for $z \in \mathcal{F}_{ST^{-1}} \cup \mathcal{F}_S \cup  \mathcal{F}_{ST}$ we have that $\tilde{f}(M_1z)=\overline{\chi\left(\frac{d_2-a_1(d_1d_2+c_1c_2)}{c_1}\right)}f_{\mathfrak{b}}(M_2z)$.
\end{lemma}
\begin{proof}
Firstly, note that $M_2 \in A$. Consider the matrix:
\begin{equation}
M'=M_2SM_1^{-1} = \begin{pmatrix} \cdot & \cdot \\ c_1c_2+d_1d_2 & \frac{d_2-a_1d_1d_2-a_1c_1c_2}{c_1}\end{pmatrix}
\end{equation}
Because $M_2, S$ and $M_1$ are all in $\sl2(\mathbb{Z})$, and $c_1c_2 \equiv -d_1d_2 \pmod{N}$, we have $M' \in \Gamma_0(N)$. It sends points in $\mathcal{F}_{M_1S}$ to points in $\mathcal{F}_{M_2}$, and so is the pullback matrix for points in the former. The same $M'$ can be used for $z \in \mathcal{F}_{ST}$ and $z \in \mathcal{F}_{ST^{-1}}$ also.
\end{proof}

We now write a given bounded element in terms of Fourier expansions. fix $M_1 \in A - \{I, S, ST, ST^{-1}\}$ and let $M_2$ and $M'$ be the corresponding matrices given by Lemma \ref{M2matrix}. Let $\mathfrak{a}$ and $\mathfrak{b}$ be the associated cusps, and let $\sigma_1$ and $\sigma_2$ be the associated cusp-normalising maps. Let $a_1, a_2, h_1, h_2$ be such that:
\begin{equation}
\sigma_1^{-1}M_1 = \begin{pmatrix}
1 & a_1 \\ 0 & h_1
\end{pmatrix},\;\;\sigma_2^{-1}M_2 = \begin{pmatrix}
1 & a_2 \\ 0 & h_1
\end{pmatrix}.
\end{equation}
Then, using the polar coordinates $z=e^{t+2\pi i \theta}$, we rewrite $f_{\mathfrak{a}}(\sigma^{-1}Mz)-\tilde{f}(Mz)$ as:
\begin{equation}
f_{\mathfrak{a}}\left(\frac{e^t\cos \theta + a_1}{h_1}, \frac{e^t \sin \theta}{h_1}\right) - \overline{\chi(M')}f_{\mathfrak{b}}\left(\frac{\cos(\pi-\theta)+a_2}{e^th_2},\frac{\sin(\pi-\theta)}{e^th_2}\right).
\end{equation}
This expression is denoted $E(t,\theta)$, and our aim is to bound $E(t,\theta)$ for all $M \in A - \{I, S, ST, ST^{-1}\}$ and $\{z \in B(\delta): |\Re(z)| \le \frac{1}{2}, |z|<1\}$. This is achieved by computing derivatives at $t_0,\theta_0$ and bounding at $t_1,\theta_1$. Letting $t_0=0$, set $F(u)=E(t_1u,\theta_0+(\theta_1-\theta_0)u)$, whereupon by Taylor's theorem:
\begin{equation}
\label{TaylorF1}
E(t_1,\theta_1)=F(1)=\sum_{i=0}^{d-1} \frac{F^{(i)}(0)}{i!}+\frac{F^{(d)}(u^*)}{d!},
\end{equation}
for some $0 \le u^* \le 1$. The $i$-th term is given by:

\begin{equation}
\label{Fiuexpression}
\frac{F^{(i)}(u)}{i!}=\sum_{r+s=i}\frac{t_1^r}{r!}\frac{(\theta_1-\theta_0)^s}{s!}\frac{\partial^{r+s}E}{\partial t^r \partial \theta^s}(t_1 u, \theta_0+(\theta_1-\theta_0)u).
\end{equation}

We take the $N+1$ points $e\left(\frac{\pi(N+j)}{3N}\right)$ for $j=0,1,...,N$ and fix $\delta$ such that for each point $z = e^{t_z + 2\pi i \theta_z} \in B(\delta)$ there exists some sample point $p=e^{t_p + 2 \pi i \theta_p}$ such that $\max(|t_z-t_p|,|\theta_z-\theta_p|) \le\frac{\pi}{6N}$. This gives the bound:
\begin{equation}
\left| \frac{F^{(i)}(0)}{i!}\right| \le \left( \frac{\pi}{6N}\right)^i \sum_{r+s=i}(r!s!)^{-1}\left| \frac{\partial^{r+s}E}{\partial t^r \partial \theta^s}(0, \theta_0)\right|.
\end{equation}
By (\ref{TaylorF1}) this bound, along with a bound for $F^{(d)}(u^*)$, bounds $f_{\mathfrak{a}}(\sigma^{-1}Mz)-\tilde{f}(Mz)$ over $\{z \in B(\delta):|\Re(z)|\le \frac{1}{2}\}$.

We now convert to rectangular coordinates, which are more convenient for computing the derivatives of Fourier expansions. Fix:
\begin{equation}
f(x,y) = f_{\mathfrak{a}}\left(\frac{e^t \cos \theta + a}{h}, \frac{e^t \sin \theta}{h}\right),
\end{equation}
where $\mathfrak{a}$ is a cusp of width $h$, and $0 \le a < h$. The derivatives are generated recursively. All derivatives are of the form:

\begin{equation}
\frac{\partial^{r+s}f}{\partial t^r \partial \theta^s}(x,y) = \sum_{k+\ell \le r+s} P(x,y;r,s,k,\ell)\frac{\partial^{k+\ell}f}{\partial x^k \partial y^\ell}(x,y),
\end{equation}
where $P=0$ whenever $k,l<0$ or $k+\ell>r+s$, $P(x,y;0,0,0,0)=1$ and all other $P$ are given by the relations:
\begin{align}
P(x,y;r+1,s,k,\ell)=&\left(x-\frac{a}{h}\right)\left(\frac{\partial}{\partial x}P(x,y;r,s,k,\ell)+P(x,y;r,s,k-1,\ell)\right) \nonumber \\
&+ y\left(\frac{\partial}{\partial y}P(x,y;r,s,k,\ell)+P(x,y;r,s,k,\ell-1)\right),  \\
P(x,y;r,s+1,k,\ell)=&-y\left(\frac{\partial}{\partial x}P(x,y;r,s,k,\ell)+P(x,y;r,s,k-1,\ell)\right) \nonumber \\
+&\left(x-\frac{a}{h}\right)\left(\frac{\partial}{\partial y}P(x,y;r,s,k,\ell)+P(x,y;r,s,k,\ell-1)\right).
\end{align}

We have now reduced the problem to one of computing derivatives of Fourier expansions of the form (\ref{FourierExpansionForm}). The $e((n+\mu_\mathfrak{a})x)$ factor is easily derived, so we focus instead on $W_{ir}(y)$. From the definition of the $K$-bessel function, we see:
\begin{equation}
W_{ir}''(y)=(1-\lambda y^{-2})W_{ir}(y).
\end{equation}
Consequently, we generate all derivatives of $W_{ir}$ recursively from $W_{ir}$ and $W_{ir}'$, as in \cite[(3.42)]{BookerVenkatesh}. It only remains to bound $|W_{ir}^{(d)}(y)|$ in order to compute $F^{(d)}(u^*)$ in (\ref{TaylorF1}). In \cite[Lemma 3.11]{BookerVenkatesh} it is shown that if $W_{ir}(z)=\sqrt{z}K_{ir}(z)$ for any $z$ with $\Re(z)>0$, then:
\begin{equation}
|W_{ir}^{(\ell)}(z)| \le \sqrt{(\pi/2)(|z|/\Re(z))}e^{-\Re(z)}.
\end{equation}
For $y \in \mathbb{R}_{>0}$, fix any $0<\epsilon \le y$ and apply Cauchy's theorem to the circle of radius $\epsilon$ around $y$. This gives:
\begin{equation}
|W_{ir}^{(\ell)}(y)| \le \sqrt{\pi} \ell! \epsilon^{-\ell}e^{y-\epsilon}.
\end{equation}
For any point $z_0$ with imaginary part $y_0$, let $y_{\min}=y_0e^{-\delta}$ be the least possible imaginary part of points in the $\delta$-neighbourhood of $z_0$. Taking the limit as $\epsilon \rightarrow y_{\min}$ shows that:
\begin{equation}
|W_{ir}^{(\ell)}(y_0)| \le \sqrt{\pi} \ell ! y_{\min}^{-\ell},
\end{equation}
and consequently:
\begin{equation}
\left|\frac{d^\ell}{dy^\ell}W_{ir}(2 \pi n y_0)\right| \le \sqrt{\pi} \ell ! y_{\min}^{-\ell}.
\end{equation}

This section has neglected discussion of $E(M)$ for $M \in \{I,S,ST,ST^{-1}\}$. When $N<3$ these cases are handled identically to the cases above. When $N \ge 3$, we follow the same process as above, but we have to take $2N+1$ sample points $\frac{e\left(\frac{\pi(N+2j)}{6N}\right)}{\sqrt{3}}$ for $j=0,1,...,2N$ instead.

We use the above method to present a certified $\Delta$-eigenvalue of a Maass cusp form of level 5 with quadratic character. As per \cite[Section 1.9]{Maass}, there are CM-forms in this space with $\Delta$-eigenvalues $\lambda$ of the form:
\begin{equation}
\lambda = \left(\frac{n \pi}{\sqrt{5}+1}\right)^2+\frac{1}{4},\;\; n \in \mathbb{Z}.
\end{equation}
Applying the method from \cite{Stromberg} away from these eigenvalues, we find the form:
\begin{center}
\begin{tabular}{r | l}
$\lambda$ & $24.199095328433893168...$\\
$a(\infty,2)$ & $-0.795198895476999501215...$\\
$a(\infty,3)$ & $-0.192808338985752020813...$\\
$a(\infty,4)$ & $-0.481481902375692542713...$\\
$a(\infty,5)$ & $-0.196583115622161865546...$\\
$a(\mathfrak{a},-n)$ & $-a(\mathfrak{a},n)$
\end{tabular}
\end{center}
Any space $\mathcal{S}(5,\left(\frac{\cdot}{5}\right),\lambda)$ containing non-CM forms would be multi-dimensional, and we use the assumption that this space is 2-dimensional (a consequence of \cite[Conjecture 3]{Dim1Conjecture}) to extract the purported Hecke eigenform in Appendix \ref{appendix}.

With these coefficients, we construct $f_\infty$ and $f_0$ meeting the requirements for Conditions \ref{conditions1}. Applying Theorem \ref{maintheorem} with the $M_0=40$, $d=45$ and $N=100$ gives Corollary \ref{theform}. It may be surprising that such a large amount of coefficients, differentiation and sample points was required to obtain certification only up to the second decimal place. The reason this is worse than the level 1 case is that our version of the $E$ function genuinely tests non-trivial automorphy relations, whereas at level 1 many of the relations fall trivially from the structure of the Fourier expansion. For example, $f_\infty(z)=\pm f_\infty(Sz)$ on the arc of points with absolute value 1, but there is no such guarantee that $f_\infty(z)=\pm f_0(\frac{Sz}{5})$ on the arc of points with absolute value $\frac{1}{\sqrt{3}}$.

\appendix

\section{Purported level 5 Maass cusp form with quadratic character}
\label{appendix}
The approximated eigenvalues for the purported Maass cusp form generated for Corollary \ref{theform} are:

\begin{center}
\begin{tabular}{r | l}
$\lambda$ & $24.1990953284330163389316822199$\\
$a(\infty,2)$ & $1.21716141180029715182640551470i$\\
$a(\infty,3)$ & $0.295119713346679445553519460316i$\\
$a(\infty,5)$ & $0.401708442188919067226691761974$\\
&$+ 0.915767616524056755556979034930i$\\
$a(\infty,7)$ & $-1.13887375514569341770109580569i$\\
$a(\infty,11)$ & $-0.0413962925778704152744983859814$\\
$a(\infty,13)$ & $-0.558344591842267393371226142461i$\\
$a(\infty,17)$ & $-0.212576664096405141444651460705i$\\
$a(\infty,19)$ & $0.608670097811209375904824033380$\\
$a(\infty,23)$ & $-1.20583190871320095924627312160i$\\
$a(\infty,29)$ & $0.162328581215640691510524606649$\\
$a(\infty,31)$ & $-0.556019364972687873912748994512$\\
$a(\infty,37)$ & $0.411889174612637537355782456909i$\\
$a(\infty,-1)$ & $-1$\\
\end{tabular}
\end{center}

\bibliography{papers}
\bibliographystyle{plain}

\end{document}